\documentclass[a4paper,reqno]{amsart}

\PII{}
\makeatletter
\def\@setcopyright{\@empty}
\makeatother

\newcommand{\Lp}{L_p[0,2\pi]}
\newcommand{\allp}{1<p<\infty}
\newcommand\seq[2]{\{#1_{#2}\}_{#2=1}^\infty}
\newcommand{\En}[1][n]{E_{#1}(f)_p}
\newcommand{\w}{\omega_k(f,t)_p}
\newcommand{\wpar}[1]{\omega_k\left(f,#1\right)_p}
\newcommand{\Np}{N(p,\theta,r,\lambda,\varphi)}

\theoremstyle{plain}
\newtheorem{theorem}{Theorem}[section]
\newtheorem{lemma}{Lemma}[section]
\newtheorem{corollary}{Corollary}[section]

\theoremstyle{remark}
\newtheorem{example}{Example}[section]
\newtheorem{remark}{Remark}[section]

\newcounter{tempenumi}

\newcounter{const}[section]
\numberwithin{const}{theorem}
\numberwithin{const}{lemma}
\numberwithin{const}{corollary}
\numberwithin{const}{example}

\makeatletter
\newcommand{\Cn}[1][]{%
  \stepcounter{const}C_{\theconst}%
  \@ifnotempty{#1}{\newcounter{#1}\setcounter{#1}{\arabic{const}}}}
\makeatother

\newcommand{\lastC}{C_{\theconst}}
\newcommand{\prevC}[1][1]{%
	{\countdef\n=255
	 \n=\theconst
	 \advance\n by-#1
	 C_{\number\n}}}

\renewcommand{\theconst}{\arabic{const}}

\begin{document}

\title[On approximations by trigonometric polynomials\dots]{%
	On approximations by trigonometric polynomials
	of classes of functions
	defined by moduli of smoothness%
}
\author[N.~Sh.\ Berisha]{Nimete Sh.~Berisha}
\address{Nimete Sh.~Berisha\\
	Faculty of Economics\\
	University of Prishtina\\
	N\"ena Terez\"e~5\\
	Prishtina\\
	Kosovo%
}
\email{nimete.berisha@uni-pr.edu}
\author[F.~M.\ Berisha]{Faton M.~Berisha}
\address{Faton M.~Berisha\\
	Faculty of Mathematics and Sciences\\
	University of Prishtina%
}
\email{faton.berisha@uni-pr.edu}
\author[M.~K.\ Potapov]{Mikhail~K.\ Potapov}
\address{Mikhail~K.\ Potapov\\
	Department of Mechanics and Mathematics\\
	Moscow State University\\
	Moscow 117234\\
	Russia}
\email{mkpotapov@mail.ru}
\author[M.~Dema]{Marjan Dema}
\address{Marjan Dema\\
	Faculty of Electrical and Computer Engineering\\
	University of Prishtina}
\email{marjan.dema@uni-pr.edu}

\date{}

\begin{abstract}
	In this paper,
	we give a characterization
	of Nikol'ski\u{\i}-Besov type classes of functions,
	given by integral representations of moduli of smoothness,
	in terms of series over the moduli of smoothness.
	Also,
	necessary and sufficient conditions
	in terms of monotone or lacunary Fourier coefficients
	for a function to belong to a such a class
	are given.
	In order to prove our results,
	we make use of certain recent
	reverse Copson- and Leindler-type inequalities.
\end{abstract}

\maketitle

\section{Introduction}

Let
\footnote{%
	The authors declare that there is no conflict of interest
	regarding the publication of this paper.%
}
$f\in\Lp$, $\allp$, be a $2\pi$-periodic function.
We say that the function~$f$ has monotone Fourier coefficients if
it has a cosine Fourier series with
\begin{displaymath}
	f(x)\sim\sum_{n=0}^\infty a_n\cos nx, \quad a_n\downarrow0.
\end{displaymath}

We say that the function~$f$ has lacunary Fourier coefficients if
\begin{displaymath}
  f(x)\sim\sum_{\nu=1}^\infty\lambda_\nu\cos\nu x,
\end{displaymath}
where
\begin{displaymath}
	\lambda_\nu=
	\begin{cases}
  	a_\mu\ge0 &\text{for $\nu=2^\mu$},\\
  	0			 &\text{for $\nu\ne2^\mu$},
	\end{cases}
\end{displaymath}
that is
\begin{displaymath}
	f(x)\sim\sum_{\mu=0}^\infty a_\mu\cos2^\mu x, \quad a_\mu\ge0.
\end{displaymath}

By~$\w$ we denote the modulus of smoothness of order~$k$
in~$L_p$ metrics of a function $f\in L_p$, $\allp$:
\begin{displaymath}
	\w=\sup_{|h|\le t}\|\Delta_h^k f\|_p,
\end{displaymath}
where
\begin{displaymath}
	\Delta_h^k f(x)=\sum_{\nu=0}^k(-1)^{k-\nu}\binom k\nu f(x+\nu h)
\end{displaymath}
is the $k$-th order shift operator.

By~$\En$ we denote the best approximation in~$L_p$ metrics of a
function $f\in L_p$, $\allp$, by means of trigonometric polynomials
whose degree is not greater than $n-1$, i.e.
\begin{displaymath}
	\En=\inf_{T_{n-1}}\|f-T_{n-1}\|_p,
\end{displaymath}
where
$T_{n-1}=\sum_{\nu=0}^{n-1}(\alpha_\nu\cos\nu x+\beta_\nu\sin\nu x)$,
$\alpha_\nu$ and~$\beta_\nu$ are arbitrary real numbers.

We say that a $2\pi$-periodic function~$f$
belongs to the Nikol'ski\u{\i}-Besov class $\Np$, $\allp$,
if the following conditions are satisfied
\begin{enumerate}
	\item $f\in\Lp$;
	\item\label{it:parameters}Numbers~$\theta$, $r$,~$\lambda$
	  belong to the interval $(0,\infty)$, and~$k$ is an integer
	  satisfying $k>r+\lambda$;
	\item The following inequality holds true
		\begin{displaymath}
			\biggl(
				\int_0^\delta t^{-r\theta-1}\w^\theta\,dt
				+\delta^{\lambda\theta}\int_\delta^1
					t^{-(r+\lambda)\theta-1}\w^\theta\,dt
			\biggr)^{1/\theta}
			\le C\varphi(\delta),
		\end{displaymath}
	\setcounter{tempenumi}{\theenumi}
\end{enumerate}
while the function~$\varphi$ satisfies the conditions
\begin{enumerate}
	\setcounter{enumi}{\thetempenumi}
	\item\label{it:phi-continuous}$\varphi$~is a non-negative
	  continuous function on $(0,1)$ and $\varphi\ne0$;
	\item For every~$\delta_1$, $\delta_2$ such that
		$0\le\delta_1\le\delta_2\le1$ holds
		$\varphi(\delta_1)\le\Cn\varphi(\delta_2)$;
	\item\label{it:phi-2}For every~$\delta$ such that
	  $0\le\delta\le\frac12$
		holds $\varphi(2\delta)\le\Cn\varphi(\delta)$,
\end{enumerate}
where constants%
\footnote{Without mentioning it explicitly,
	we will consider all the constants positive.%
}%
~$C$, $\prevC$ and~$\lastC$
do not depend on~$\delta_1$, $\delta_2$ and~$\delta$.

A more detailed approach to the classes $\Np$
is given in~\cite{lakovic:mat-87} and~\cite{tikhonov:etna-05}
(see also~\cite{besov-i-n:integral}%
).
In the paper,
we give a characterization of $\Np$
classes of functions
in terms of series over their moduli of smoothness.
Then we give the necessary and sufficient conditions
in terms of monotone or lacunary Fourier coefficients
for a function $f\in\Lp$ to belong to a class $\Np$.
In the process of proving the results,
we make use of certain recent
reverse $l_p$-type inequalities~\cite{p-berisha-b-k:mia-15},
closely related to Copson's	and Leindler's inequalities.

Finally, by making use of our results,
we construct an example of a function having a lacunary Fourier series,
which shows that $\Np$ classes are properly embedded
between the appropriate Nikol'ski\u{\i} classes
and Besov classes.

\section{Statement of Results}

Now we formulate our results.

\begin{theorem}\label{th:Np-w}
	A function~$f$ belongs to the class $\Np$
	if and only if%
	\footnote{Here and below we assume that the parameters~$\theta$,
		$r$, $\lambda$ and~$k$
		satisfy the condition~\ref{it:parameters},
		and the function~$\varphi$ satisfies the conditions
		\ref{it:phi-continuous}--\ref{it:phi-2}
		of the definition of the class $\Np$.}
	\begin{multline}\label{eq:Np-w}
		\biggl(
			\sum_{\nu=n+1}^\infty \wpar{\frac1\nu}^\theta \nu^{r\theta-1}
			+n^{-\lambda\theta}\sum_{\nu=1}^n
				\wpar{\frac1\nu}^\theta \nu^{(r+\lambda)\theta-1}
		\biggr)^{1/\theta}\\
		\le C\varphi\left(\frac1n\right),
	\end{multline}
	where constant~$C$ does not depend on~$n$.
\end{theorem}

\begin{theorem}\label{th:Np-monotone}
	For a function $f\in\Lp$, $\allp$,
	such that
	\begin{equation}\label{eq:f}
		f(x)\sim\sum_{\nu=1}^\infty a_\nu\cos\nu x, \quad a_\nu\downarrow0,
	\end{equation}
	to belong to the class $\Np$
	it is necessary and sufficient
	that its Fourier coefficients satisfy the condition
	\begin{displaymath}
		\biggl(
			\sum_{\nu=n+1}^\infty a_\nu^\theta \nu^{r\theta+\theta-\theta/p-1}
			+n^{-\lambda\theta}\sum_{\nu=1}^n a_\nu^\theta
				\nu^{r\theta+\lambda\theta+\theta-\theta/p-1}
		\biggr)^{1/\theta}
		\le C\varphi\left(\frac 1n\right),
	\end{displaymath}
	where constant~$C$ does not depend on~$n$.
\end{theorem}

Regarding Theorem~\ref{th:Np-w},
a very interesting open question remains
its analogue for functions with general monotone Fourier coefficients,
generalized in the sense
of~\cite{tikhonov:analysis-07,liflyand-t:math-nachr-11}.

\begin{corollary}
	Put $\varphi(\delta)=\delta^\alpha$, $0<\alpha<\lambda$,
	in the definition of the class $\Np$,
	we obtain~\cite{lakovic:mat-87}
	the Nikol'ski\u{\i} class~$H_p^{r+\alpha}$.
	Thus Theorems~\ref{th:Np-w} and~\ref{th:Np-monotone}
	give the single coefficient condition
	\begin{displaymath}
		a_\nu\le\frac C{\nu^{r+\alpha+1-\frac1p}},
	\end{displaymath}
	for $f\in H_p^{r+\alpha}$,
	given in~\cite{konyushkov:izv-57},
	where the function~$f$ is given by~\eqref{eq:f}.
\end{corollary}

\begin{corollary}
	If $\varphi(\delta)\ge C$,
	then we obtain~\cite{lakovic:mat-87}
	the Besov class~$B_p^{\theta r}$.
	Thus Theorems~\ref{th:Np-w} and~\ref{th:Np-monotone}
	give the necessary and sufficient condition
	\begin{displaymath}
		\sum_{\nu=1}^\infty a_\nu^\theta \nu^{r\theta+\theta-\theta/p-1}
		<\infty
	\end{displaymath}
	for $f\in B_p^{\theta r}$,
	given in~\cite{potapov-b:publ-79},
	where the function~$f$ is given by~\eqref{eq:f}.
\end{corollary}

\begin{theorem}\label{th:Np-lacunary}
	For a function $f\in L_p$, $\allp$, such that
	\begin{equation}\label{eq:f-lacunary}
		f(x)\sim\sum_{\nu=1}^\infty\lambda_\nu\cos\nu x,
	\end{equation}
	and
	\begin{displaymath}
		\lambda_\nu=
		\begin{cases}
			a_\mu\ge0 &\text{for $\nu=2^\mu$},\\
			0			 &\text{for $\nu\ne2^\mu$},
		\end{cases}
	\end{displaymath}
	to belong to the class $\Np$ it is necessary and sufficient that its
	Fourier coefficients satisfy the condition%
	\footnote{Here and below we assume that the parameters~$\theta$, $r$,
		$\lambda$ and~$k$ satisfy the condition~\ref{it:parameters},
		and the function~$\varphi$ satisfies the conditions
		\ref{it:phi-continuous}--\ref{it:phi-2}.}
	\begin{displaymath}
		\biggl(
			\sum_{\nu=m+1}^\infty \lambda_\nu^\theta \nu^{r\theta}
			+m^{-\lambda\theta}
			  \sum_{\nu=1}^m \lambda_\nu^\theta \nu^{(r+\lambda)\theta}
		\biggr)^{1/\theta}
		\le C\varphi\left(\frac 1m\right),
	\end{displaymath}
	where constant~$C$ does not depend on~$m$.
\end{theorem}

\begin{corollary}
	Putting $\varphi(\delta)=\delta^\alpha$, $0<\alpha<\lambda$, in the
	definition of the class $\Np$, we obtain~\cite{lakovic:mat-87}
	the Nikol'ski\u{\i} class~$H_p^{r+\alpha}$. Thus
	Theorem~\ref{th:Np-lacunary} gives the single coefficient condition
	\begin{displaymath}
		a_\mu\le C2^{-\mu(r+\alpha)}
	\end{displaymath}
	for $f\in H_p^{r+\alpha}$,
	where the
	function~$f$ is given by~\eqref{eq:f-lacunary}.
\end{corollary}

\begin{corollary}
	If $\varphi(\delta)=C$, then we obtain~\cite{lakovic:mat-87}
	the Besov class~$B_p^{\theta r}$. Thus
	Theorem~\ref{th:Np-lacunary} gives the necessary
	and sufficient condition
	\begin{displaymath}
		\sum_{\mu=1}^\infty a_\mu^\theta 2^{\mu r\theta}<\infty
	\end{displaymath}
	for $f\in B_p^{\theta r}$, given in~\cite{potapov-b:publ-79}, where
	the function~$f$ is given by~\eqref{eq:f-lacunary}.
\end{corollary}

\begin{example}
	Let
	\begin{displaymath}
		f(x)\sim\sum_{\mu=0}^\infty a_\mu\cos2^\mu x,
	\end{displaymath}
	where are
	\begin{displaymath}
		a_\mu=2^{-\mu r}(\mu+1)^{-(\alpha+1/\theta)}, \quad \alpha>0.
	\end{displaymath}
	Then, we have
	\begin{displaymath}
		\Cn n^{-\alpha}
		\le\biggl(
				\sum_{\mu=n+1}^\infty a_\mu^\theta 2^{\mu r\theta}
			\biggr)^{1/\theta}
		\le\Cn n^{-\alpha}
	\end{displaymath}
	and
	\begin{displaymath}
		\Cn n^{-(\alpha+1/\theta)}
		\le\biggl(
				2^{-n\lambda\theta}
				  \sum_{\mu=0}^n a_\mu^\theta 2^{\mu(r+\lambda)\theta}
			\biggr)^{1/\theta}
		\le\Cn n^{-(\alpha+1/\theta)},
	\end{displaymath}
	thus implying \textup(see the proof of
	Theorem~\ref{th:Np-lacunary}\textup)
	$f\in\Np$ for
	$\varphi(\delta)=\left(\ln\frac1\delta\right)^{-\alpha}$.
	This means that classes~$N$ are classes of embedding between
	classes~$H$ and~$B$.
\end{example}

\section{Auxiliary statements}

In order to establish our results,
we use the following lemmas.

\begin{lemma}\label{lm:jensen}
	Let $0<\alpha<\beta<\infty$ and $a_\nu\ge0$. The following
	inequality holds true
	\begin{displaymath}
		\biggl(\sum_{\nu=1}^n a_\nu^\beta\biggr)^{1/\beta}
		\le\biggl(\sum_{\nu=1}^n a_\nu^\alpha\biggr)^{1/\alpha}.
	\end{displaymath}
\end{lemma}

Proof of the lemma
is due to Jensen~\cite[p.~43]{hardy-l-p:inequalities}.

\begin{lemma}\label{lm:lp}
	Let $\seq a\nu$ be a sequence of non-negative numbers, $\alpha>0$,
	$\lambda$ a real number,
	$m$ and~$n$ positive integers such that $m<n$.
	Then
	\begin{enumerate}
	\item for $1\le p<\infty$ the following equalities hold
		\begin{displaymath}
			\sum_{\mu=m}^n \mu^{\alpha-1}
				\biggl(\sum_{\nu=\mu}^n a_\nu \nu^\lambda\biggr)^p
			\le\Cn\sum_{\mu=m}^n \mu^{\alpha-1}(a_\mu \mu^{\lambda+1})^p,
		\end{displaymath}
		\begin{displaymath}
			\sum_{\mu=m}^n \mu^{-\alpha-1}
				\biggl(\sum_{\nu=m}^\mu a_\nu \nu^\lambda\biggr)^p
			\le\Cn\sum_{\mu=m}^n \mu^{-\alpha-1}(a_\mu \mu^{\lambda+1})^p;
		\end{displaymath}
	\item for $0<p\le1$ the following equalities hold
		\begin{displaymath}
			\sum_{\mu=m}^n \mu^{\alpha-1}
				\biggl(\sum_{\nu=\mu}^n a_\nu \nu^\lambda\biggr)^p
			\ge\Cn\sum_{\mu=m}^n \mu^{\alpha-1}(a_\mu \mu^{\lambda+1})^p,
		\end{displaymath}
		\begin{displaymath}
			\sum_{\mu=m}^n \mu^{-\alpha-1}
				\biggl(\sum_{\nu=m}^\mu a_\nu \nu^\lambda\biggr)^p
			\ge\Cn\sum_{\mu=m}^n \mu^{-\alpha-1}(a_\mu \mu^{\lambda+1})^p,
		\end{displaymath}
	\end{enumerate}
	where constants~$\prevC[3]$, $\prevC[2]$, $\prevC$ and~$\lastC$
	depend only on numbers~$\alpha$, $\lambda$ and~$p$,
	and do not depend on~$m$, $n$
	as well as on the sequence $\seq a\nu$.
\end{lemma}

Proof of the lemma is given in~\cite[p.~308]{hardy-l-p:inequalities}.

Lemmas~\ref{lm:lp-converse} and~\ref{lm:lp-asymp} that follow
state certain $l_p$-type inequalities
which are reversed to the ones given in Lemma~\ref{lm:lp}
and closely related to Copson's	and Leindler's inequalities
(see, e.g.,
\cite{%
	copson:london-28,leindler:acta-70,leindler:jipam-00,
	tikhonov-z:springer-14%
}%
).

We write $a_\nu\downarrow$
if $\seq a\nu$ is a monotone-decreasing sequence of non-negative numbers,
i.e.\ if $a_\nu\ge a_{\nu+1}\ge0$ $(\nu=1,2,\dotsc)$.

\begin{lemma}\label{lm:lp-converse}
	Let $a_\nu\downarrow$, $\alpha>0$, $\lambda$ a real number,
	$m$ and~$n$ positive integers.
	Then
	\begin{enumerate}
	\item for $1\le p<\infty$, $n\ge16m$ the following equalities hold
		\begin{displaymath}
			\sum_{\mu=m}^n \mu^{\alpha-1}
				\biggl(\sum_{\nu=\mu}^n a_\nu \nu^\lambda\biggr)^p
			\ge\Cn\sum_{\mu=8m}^n \mu^{\alpha-1}(a_\mu \mu^{\lambda+1})^p,
		\end{displaymath}
		\begin{displaymath}
			\sum_{\mu=m}^n \mu^{-\alpha-1}
				\biggl(\sum_{\nu=m}^\mu a_\nu \nu^\lambda\biggr)^p
			\ge\Cn\sum_{\mu=4m}^n \mu^{-\alpha-1}(a_\mu \mu^{\lambda+1})^p;
		\end{displaymath}
	\item for $0<p\le1$, $n\ge4m$ the following equalities hold
		\begin{displaymath}
			\sum_{\mu=4m}^n \mu^{\alpha-1}
				\biggl(\sum_{\nu=\mu}^n a_\nu \nu^\lambda\biggr)^p
			\le\Cn\sum_{\mu=m}^n \mu^{\alpha-1}(a_\mu \mu^{\lambda+1})^p,
		\end{displaymath}
		\begin{displaymath}
			\sum_{\mu=4m}^n \mu^{-\alpha-1}
				\biggl(\sum_{\nu=4m}^\mu a_\nu \nu^\lambda\biggr)^p
			\le\Cn\sum_{\mu=m}^n \mu^{-\alpha-1}(a_\mu \mu^{\lambda+1})^p,
		\end{displaymath}
	\end{enumerate}
	where constants~$\prevC[3]$, $\prevC[2]$, $\prevC$ and~$\lastC$
	depend only on numbers~$\alpha$, $\lambda$ and~$p$,
	and do not depend on~$m$, $n$
	as well as on the sequence $\seq a\nu$.
\end{lemma}

Proof of the lemma is given in~\cite{p-berisha-b-k:mia-15}.

\begin{lemma}\label{lm:lp-asymp}
	Let $a_\nu\downarrow$, $\alpha>0$, $\lambda$ a real number,
	$m$ and~$n$ positive integers.
	For $0<p<\infty$ the following inequalities hold
	\begin{displaymath}
		\Cn\sum_{\mu=1}^n \mu^{\alpha-1}(a_\mu \mu^{\lambda+1})^p
		\le\sum_{\mu=1}^n \mu^{\alpha-1}
			\biggl(\sum_{\nu=\mu}^n a_\nu \nu^\lambda\biggr)^p
		\le\Cn\sum_{\mu=1}^n \mu^{\alpha-1}(a_\mu \mu^{\lambda+1})^p,
	\end{displaymath}
	\begin{displaymath}
		\Cn\sum_{\mu=1}^n \mu^{-\alpha-1}(a_\mu \mu^{\lambda+1})^p
		\le\sum_{\mu=1}^n \mu^{-\alpha-1}
			\biggl(\sum_{\nu=1}^\mu a_\nu \nu^\lambda\biggr)^p
		\le\Cn\sum_{\mu=1}^n \mu^{-\alpha-1}(a_\mu \mu^{\lambda+1})^p,
	\end{displaymath}
	where constants~$\prevC[3]$, $\prevC[2]$, $\prevC$ and~$\lastC$
	depend only on numbers~$\alpha$, $\lambda$ and~$p$,
	and do not depend on~$m$, $n$
	as well as on the sequence $\seq a\nu$.
\end{lemma}

The lemma is also proved in~\cite{p-berisha-b-k:mia-15}.

\begin{lemma}\label{lm:anu-w}
	Let $f\in\Lp$ for a fixed~$p$ from the interval $1<p<\infty$
	and let
	\begin{displaymath}
		f(x)\sim\sum_{\nu=1}^\infty a_\nu\cos\nu x, \quad a_\nu\downarrow0.
	\end{displaymath}
	The following inequalities hold
	\begin{multline*}
		\Cn\frac1{n^k}
					\biggl(\sum_{\nu=1}^n a_\nu^p \nu^{(k+1)p-2}\biggr)^{1/p}
				+\biggl(\sum_{\nu=n+1}^\infty a_\nu^p \nu^{p-2}\biggr)^{1/p}
			\le\wpar{\frac1n}\\
		\le\Cn\frac1{n^k}
					\biggl(\sum_{\nu=1}^n a_\nu^p \nu^{(k+1)p-2}\biggr)^{1/p}
				+\biggl(\sum_{\nu=n+1}^\infty a_\nu^p \nu^{p-2}\biggr)^{1/p},
	\end{multline*}
	where constants~$\prevC$ and~$\lastC$
	do not depend on~$n$ and~$f$.
\end{lemma}

The lemma is proved in~\cite{potapov-b:publ-79}.

\begin{lemma}\label{lm:Np-E}
	A function~$f$ belongs to the class $\Np$ if and only if
	\begin{displaymath}
		\biggl(
			\sum_{\mu=n+1}^\infty 2^{\mu r\theta}\En[2^\mu]^\theta
			+2^{-n\lambda\theta}
				\sum_{\mu=0}^n 2^{\mu(r+\lambda)\theta}\En[2^\mu]^\theta
		\biggr)^{1/\theta}
		\le C\varphi\left(\frac 1{2^n}\right),
	\end{displaymath}
	where constant~$C$ does not depend on~$n$.
\end{lemma}

Proof of the lemma is given in~\cite{lakovic:mat-87}.

\begin{lemma}\label{lm:zygmund}
	Let $f\in L_p$, $\allp$, and
	\begin{displaymath}
		f(x)\sim\sum_{\mu=0}^\infty a_\mu\cos2^\mu x, \quad a_\mu\ge0.
	\end{displaymath}
	The following inequalities hold
	\begin{displaymath}
		\Cn\biggl(\sum_{\mu=0}^\infty a_\mu^2\biggr)^{1/2}\le\|f\|_p
		\le\Cn\biggl(\sum_{\mu=0}^\infty a_\mu^2\biggr)^{1/2},
	\end{displaymath}
	where constants~$\lastC$ and~$\prevC$ do not depend on~$f$.
\end{lemma}

Proof of the lemma is due to
Zygmund~\cite[vol.~I, p.~326]{zygmund:trigonometric}.

\begin{corollary}\label{cr:zygmund}
	Lemma~\ref{lm:zygmund} yields the following
	estimate
	\begin{displaymath}
		\Cn\biggl(\sum_{\mu=n}^\infty a_\mu^2\biggr)^{1/2}\le\En[2^n]
		\le\Cn\biggl(\sum_{\mu=n}^\infty a_\mu^2\biggr)^{1/2},
	\end{displaymath}
	where constants~$\lastC$ and~$\prevC$ do not depend on~$n$ and~$f$.
\end{corollary}

\section{Proofs}

Now we prove our results.

\begin{proof}[Proof of Theorem~\ref{th:Np-w}]
	\setcounter{const}{0}
	Put
	\begin{displaymath}
		I_1=\int_0^\frac1{n+1} t^{-r\theta-1}\w^\theta\,dt, \quad
		I_2=\int_\frac1{n+1}^1 t^{-(r+\lambda)\theta-1}\w^\theta\,dt.
	\end{displaymath}
	We have~\cite[p.~55]{hardy-l-p:inequalities}
	\begin{multline*}
		I_1=\int_0^\frac1{n+1} t^{-r\theta-1}\w^\theta\,dt
			=\sum_{\nu=n+1}^\infty
				\int_\frac1{\nu+1}^\frac1\nu t^{-r\theta-1}\w^\theta\,dt\\
		\le\sum_{\nu=n+1}^\infty \wpar{\frac1\nu}^\theta
				\int_\frac1{\nu+1}^\frac1\nu t^{-r\theta-1}\,dt
			\le\Cn\sum_{\nu=n+1}^\infty \wpar{\frac1\nu}^\theta
				\nu^{r\theta-1}
	\end{multline*}
	and,
	taking into account properties
	of modulus of smoothness~\cite{timan:approximation},
	\begin{displaymath}
		I_1\ge\sum_{\nu=n+1}^\infty \wpar{\frac1{\nu+1}}^\theta
				\int_\frac1{\nu+1}^\frac1\nu t^{-r\theta-1}\,dt
			\ge\Cn\sum_{\nu=n+1}^\infty \wpar{\frac1\nu}^\theta
				\nu^{r\theta-1}.
	\end{displaymath}

	In an analogous way we estimate
	\begin{displaymath}
		I_2\le\sum_{\nu=1}^n \wpar{\frac1\nu}^\theta
				\int_\frac1{\nu+1}^\frac1\nu t^{-(r+\lambda)\theta-1}\,dt
			\le\Cn\sum_{\nu=1}^n \wpar{\frac1\nu}^\theta
				\nu^{(r+\lambda)\theta-1}
	\end{displaymath}
	and
	\begin{displaymath}
		I_2\ge\sum_{\nu=1}^n \wpar{\frac1{\nu+1}}^\theta
				\int_\frac1{\nu+1}^\frac1\nu t^{-(r+\lambda)\theta-1}\,dt
			\ge\Cn\sum_{\nu=1}^n \wpar{\frac1\nu}^\theta
				\nu^{(r+\lambda)\theta-1}.
	\end{displaymath}

	Let $f\in\Np$.
	For a positive integer~$n$
	we put $\delta=\frac1{n+1}$.
	Then we have
	\begin{multline*}
		I^\theta=I_1+\delta^{\lambda\theta}I_2\\
		\ge\Cn\biggl(
			\sum_{\nu=n+1}^\infty \wpar{\frac1\nu}^\theta \nu^{r\theta-1}
			+n^{-\lambda\theta}\sum_{\nu=1}^n \wpar{\frac1\nu}^\theta
				\nu^{(r+\lambda)\theta-1}
		\biggr).
	\end{multline*}
	Hence we obtain
	\begin{multline*}
		J=\biggl(
				\sum_{\nu=n+1}^\infty \wpar{\frac1\nu}^\theta \nu^{r\theta-1}
				+n^{-\lambda\theta}\sum_{\nu=1}^n \wpar{\frac1\nu}^\theta
					\nu^{(r+\lambda)\theta-1}
			\biggr)^{1/\theta}\\
		\le\Cn I\le\Cn\varphi(\delta)=\lastC\varphi\left(\frac1{n+1}\right)
			\le\Cn\varphi\left(\frac1n\right),
	\end{multline*}
	which proves inequality~\eqref{eq:Np-w}.

	Now we suppose that inequality~\eqref{eq:Np-w} holds.
	For $\delta\in(0,1)$ we choose the positive integer~$n$
	satisfying $\frac1{n+1}<\delta\le\frac1n$.
	Then,
	taking into consideration the estimates from above
	for~$I_1$ and~$I_2$
	we have
	\begin{multline*}
		I^\theta=\int_0^\frac1{n+1} t^{-r\theta-1}\w^\theta\,dt
			+\int_\frac1{n+1}^\delta t^{-r\theta-1}\w^\theta\,dt\\
		+\delta^{\lambda\theta}
				\int_\delta^1 t^{-(r+\lambda)\theta-1}\w^\theta\,dt
			\le I_1+\delta^{\lambda\theta}I_2\\
		\le\Cn\biggl(
			\sum_{\nu=n+1}^\infty \wpar{\frac1\nu}^\theta \nu^{r\theta-1}
			+n^{-\lambda\theta}\sum_{\nu=1}^n \wpar{\frac1\nu}^\theta
				\nu^{(r+\lambda)\theta-1}
		\biggr).
	\end{multline*}
	Hence
	\begin{displaymath}
		I\le\Cn J\le\Cn\varphi\left(\frac1n\right)
		\le\Cn\varphi\left(\frac1{2n}\right)\le\Cn\varphi(\delta),
	\end{displaymath}
	implying $f\in\Np$.

	Proof of Theorem~\ref{th:Np-w} is completed.
\end{proof}

\begin{proof}[Proof of Theorem~\ref{th:Np-monotone}]
	\setcounter{const}{0}
	Theorem~\ref{th:Np-w} implies that the condition $f\in\Np$
	is equivalent to the condition
	\begin{displaymath}
		\sum_{\nu=n+1}^\infty \wpar{\frac1\nu}^\theta \nu^{r\theta-1}
			+n^{-\lambda\theta}\sum_{\nu=1}^n \wpar{\frac1\nu}^\theta
				\nu^{(r+\lambda)\theta-1}
		\le\Cn\varphi\left(\frac 1n\right)^\theta,
	\end{displaymath}
	where constant~$\lastC$ does not depend on~$n$.
	Lemma~\ref{lm:anu-w} yields that the last estimate
	is equivalent to the estimate
	\begin{multline*}
		\sum_{\nu=n+1}^\infty \nu^{(r-k)\theta-1}
				\biggl(
					\sum_{\mu=1}^\nu a_\mu^p \mu^{(k+1)p-2}
				\biggr)^{\theta/p}
			+\sum_{\nu=n+1}^\infty \nu^{r\theta-1}
				\biggl(
				  \sum_{\mu=\nu}^\infty a_\mu^p \mu^{p-2}
				\biggr)^{\theta/p}\\
		+n^{-\lambda\theta}\sum_{\nu=1}^n \nu^{(r+\lambda-k)\theta-1}
				\biggl(
				  \sum_{\mu=1}^\nu a_\mu^p \mu^{(k+1)p-2}
				\biggr)^{\theta/p}\\
		+n^{-\lambda\theta}\sum_{\nu=1}^n \nu^{(r+\lambda)\theta-1}
				\biggl(
					\sum_{\mu=\nu}^\infty a_\mu^p \mu^{p-2}
				\biggr)^{\theta/p}
			\le\Cn\varphi\left(\frac 1n\right)^\theta,
	\end{multline*}
	where constant~$\lastC$ does not depend on~$n$.
	Hence,
	if we denote the terms on the left-hand side of the inequality
	by~$J_1$, $J_2$, $J_3$ and~$J_4$ respectively,
	then condition $f\in\Np$ is equivalent to the condition
	\begin{equation}\label{eq:J-phi}
		J_1+J_2+J_3+J_4\le\lastC\varphi\left(\frac 1n\right)^\theta.
	\end{equation}

	Now we estimate the terms~$J_1$, $J_2$, $J_3$ and~$J_4$
	from below and above
	by means of expression taking part in the condition of the theorem.

	First we estimate~$J_1$ and~$J_2$ from below.
	We have
	\begin{multline*}
		J_1=\sum_{\nu=n+1}^\infty \nu^{(r-k)\theta-1}
				\biggl(
				  \sum_{\mu=1}^\nu a_\mu^p \mu^{(k+1)p-2}
				\biggr)^{\theta/p}\\
		\ge\sum_{\nu=n+1}^\infty \nu^{-(k-r)\theta-1}
				\biggl(
				  \sum_{\mu=n+1}^\nu a_\mu^p \mu^{(k+1)p-2}
				\biggr)^{\theta/p}.
	\end{multline*}
	For $k-r>0$,
	making use of Lemmas~\ref{lm:lp} and~\ref{lm:lp-converse}
	we obtain
	\begin{multline}\label{eq:J1ge}
		J_1\ge\Cn\sum_{\nu=4(n+1)}^\infty \nu^{-(k-r)\theta-1}
			(a_\nu^p \nu^{(k+1)p-2 \nu})^{\theta/p}\\
		=\lastC\sum_{\nu=4(n+1)}^\infty a_\nu^\theta
			\nu^{r\theta+\theta-\theta/p-1}.
	\end{multline}
	In an analogous way,
	for $r\theta>0$ we get
	\begin{equation}\label{eq:J2ge}
		J_2=\sum_{\nu=n+1}^\infty \nu^{r\theta-1}
			\biggl(
				\sum_{\mu=\nu}^\infty a_\mu^p \mu^{p-2}
			\biggr)^{\theta/p}
		\ge\Cn\sum_{\nu=8(n+1)}^\infty a_\nu^\theta
		  \nu^{r\theta+\theta-\theta/p-1}.
	\end{equation}

	We estimate the term~$J_2$ from above:
	\begin{equation}\label{eq:J2le}
		J_2\le\Cn\sum_{\nu=\left[\frac{n+1}4\right]}^\infty \nu^{r\theta-1}
			(a_\nu^p \nu^{p-2 \nu})^{\theta/p}
		=\lastC\sum_{\nu=\left[\frac{n+1}4\right]}^\infty a_\nu^\theta
			\nu^{r\theta+\theta-\theta/p-1}.
	\end{equation}

	For~$J_1$ we have
	\begin{multline*}
		J_1\le\Cn\biggl(\sum_{\nu=n+1}^\infty \nu^{-(k-r)\theta-1}
				\biggl(
				  \sum_{\mu=n+1}^\nu a_\mu^p \mu^{(k+1)p-2}
				\biggr)^{\theta/p}\\
		+\sum_{\nu=n+1}^\infty \nu^{-(k-r)\theta-1}
				\biggl(
				  \sum_{\mu=1}^n a_\mu^p \mu^{(k+1)p-2}
				\biggr)^{\theta/p}\biggr),
	\end{multline*}
	and applying once more Lemmas~\ref{lm:lp} and~\ref{lm:lp-converse}
	we obtain
	\begin{equation}\label{eq:J1}
		J_1\le\Cn\sum_{\nu=\left[\frac{n+1}4\right]}^\infty a_\nu^\theta
				\nu^{r\theta+\theta-\theta/p-1}
			+n^{-(k-r)\theta}
				\biggl(
					\sum_{\mu=1}^n a_\mu^p \mu^{(k+1)p-2}
				\biggr)^{\theta/p}.
	\end{equation}
	Put
	\begin{displaymath}
		I_1=n^{-(k-r)\theta}\sum_{\mu=1}^n a_\mu^p \mu^{(k+1)p-2}.
	\end{displaymath}
	Then for
	\begin{displaymath}
		I_2=I_1n^{(k-r)\theta},
	\end{displaymath}
	taking into account that $(k+1)p-2\ge0$ and $a_\nu\downarrow0$
	we get
	\begin{multline*}
		I_2=\sum_{\mu=1}^n a_\mu^p \mu^{(k+1)p-2}
		\le\sum_{\mu=1}^{\left[\frac n2\right]}
				a_\mu^p \mu^{(k+1)p-2}
			+a_{\left[\frac n2\right]+1}^p
				\sum_{\mu=\left[\frac n2\right]+1}^n \mu^{(k+1)p-2}\\
		\le\sum_{\mu=1}^{\left[\frac n2\right]}
				a_\mu^p \mu^{(k+1)p-2}
			+\Cn n^{(k+1)p-1}a_{\left[\frac n2\right]+1}^p
		\le\Cn\sum_{\mu=1}^{\left[\frac n2\right]} a_\mu^p \mu^{(k+1)p-2}.
	\end{multline*}
	Since $k-r-\lambda>0$,
	we have
	\begin{multline*}
		I_1^{\theta/p}\le\Cn n^{-(k-r)\theta}
			\biggl(
				\sum_{\mu=1}^{\left[\frac n2\right]} a_\mu^p \mu^{(k+1)p-2}
			\biggr)^{\theta/p}\\
		\le\Cn n^{-\lambda\theta}\sum_{\nu=\left[\frac n2\right]}^n
			\nu^{-(k-r-\lambda)\theta-1}
			\biggl(
				\sum_{\mu=1}^\nu a_\mu^p \mu^{(k+1)p-2}
			\biggr)^{\theta/p}\\
		\le\lastC n^{-\lambda\theta}\sum_{\nu=1}^n
			\nu^{-(k-r-\lambda)\theta-1}
			\biggl(
				\sum_{\mu=1}^\nu a_\mu^p \mu^{(k+1)p-2}
			\biggr)^{\theta/p}.
	\end{multline*}
	Applying Lemma~\ref{lm:lp-asymp} we obtain
	\begin{multline*}
		I_1^{\theta/p}\le\Cn n^{-\lambda\theta}
			\sum_{\nu=1}^n	\nu^{-(k-r-\lambda)\theta-1}
			(a_\nu^p \nu^{(k+1)p-2} \nu)^{\theta/p}\\
		=\lastC n^{-\lambda\theta} \sum_{\nu=1}^n a_\nu^\theta
			\nu^{(r+\lambda)\theta+\theta-\theta/p-1}.
	\end{multline*}
	From~\eqref{eq:J1} it follows that
	\begin{equation}\label{eq:J1le}
		J_1\le\Cn\biggl(
			\sum_{\nu=\left[\frac{n+1}4\right]}^\infty a_\nu^\theta
					\nu^{r\theta+\theta-\theta/p-1}
				+n^{-\lambda\theta}\sum_{\nu=1}^n a_\nu^\theta
					\nu^{(r+\lambda)\theta+\theta-\theta/p-1}
		\biggr).
	\end{equation}

	This way,
	inequalities~\eqref{eq:J1ge}, \eqref{eq:J2ge},
	\eqref{eq:J2le} and~\eqref{eq:J1le}
	yield
	\begin{multline}\label{eq:J1-J2}
		\Cn\sum_{\nu=8(n+1)}^\infty
				a_\nu^\theta
				\nu^{r\theta+\theta-\theta/p-1}
			\le J_1+J_2\\
		\le\Cn\biggl(
				\sum_{\nu=\left[\frac{n+1}4\right]}^\infty
					a_\nu^\theta
					\nu^{r\theta+\theta-\theta/p-1}
				+n^{-\lambda\theta}\sum_{\nu=1}^n a_\nu^\theta
					\nu^{(r+\lambda)\theta+\theta-\theta/p-1}
			\biggr).
	\end{multline}

	Now we estimate~$J_3$ and~$J_4$. Put
	\begin{displaymath}
		A_1=n^{\lambda\theta}J_3
		=\sum_{\nu=1}^n \nu^{(r+\lambda-k)\theta-1}
			\biggl(
				\sum_{\mu=1}^\nu a_\mu^p \mu^{(k+1)p-2}
			\biggr)^{\theta/p}
	\end{displaymath}
	and
	\begin{displaymath}
		A_2
		=n^{\lambda\theta}J_4=\sum_{\nu=1}^n \nu^{(r+\lambda)\theta-1}
			\biggl(
				\sum_{\mu=\nu}^\infty a_\mu^p \mu^{p-2}
			\biggr)^{\theta/p},
	\end{displaymath}
	applying Lemma~\ref{lm:lp-asymp} for $r+\lambda-k<0$
	we get
	\begin{equation}\label{eq:A1}
		A_1
		\le\Cn\sum_{\nu=1}^n a_\nu^\theta
			\nu^{(r+\lambda)\theta+\theta-\theta/p-1}.
	\end{equation}

	We estimate~$A_2$ in an analogous way:
	\begin{multline}\label{eq:A2}
		A_2\le\Cn\biggl(
				\sum_{\nu=1}^n \nu^{(r+\lambda)\theta-1}
					\biggl(
					  \sum_{\mu=\nu}^n a_\mu^p \mu^{p-2}
					\biggr)^{\theta/p}\\
		+\sum_{\nu=1}^n \nu^{(r+\lambda)\theta-1}
				\biggl(
				  \sum_{\mu=n+1}^\infty a_\mu^p \mu^{p-2}
				\biggr)^{\theta/p}
			\biggr)\\
		\le\Cn\biggl(
				\sum_{\nu=1}^n a_\nu^\theta
					\nu^{(r+\lambda)\theta+\theta-\theta/p-1}
				+n^{(r+\lambda)\theta}
					\biggl(
					  \sum_{\mu=n+1}^\infty a_\mu^p \mu^{p-2}
					\biggr)^{\theta/p}
			\biggr).
	\end{multline}

	We estimate the series
	\begin{displaymath}
		B=\biggl(\sum_{\mu=n+1}^\infty a_\mu^p \mu^{p-2}\biggr)^{\theta/p}.
	\end{displaymath}

	First let $\frac\theta p>1$.
	Applying H\"older inequality we have
	\begin{multline*}
		\sum_{\mu=n+1}^\infty
			a_\mu^p \mu^{p-2}
		\le\biggl(
			\sum_{\mu=n+1}^\infty(a_\mu^p \mu^{p-1+rp-p/\theta})^{\theta/p}
		\biggr)^{p/\theta}\\
		\times\biggl(
			\sum_{\mu=n+1}^\infty
				\mu^{-(rp-p/\theta+1)\theta/(\theta-p)}
		\biggr)^{(\theta-p)/\theta}.
	\end{multline*}
	Since
	$\bigl(rp-\frac p\theta+1\bigr)\frac\theta{\theta-p}
	=rp\frac\theta{\theta-p}+1>1$,
	we get
	\begin{displaymath}
		\sum_{\mu=n+1}^\infty a_\mu^p \mu^{p-2}
		\le\Cn n^{-rp}\biggl(
			\sum_{\mu=n+1}^\infty a_\mu^\theta
					\mu^{\theta-\theta/p+r\theta-1}
		\biggr)^{p/\theta}.
	\end{displaymath}
	So, for $\frac\theta p>1$ we have proved that
	\begin{displaymath}
		B\le\Cn n^{-r\theta}\sum_{\mu=n+1}^\infty a_\mu^\theta
			\mu^{r\theta+\theta-\theta/p-1}.
	\end{displaymath}

	Let $\frac\theta p\le1$.
	For given~$n$ we choose the positive integer~$N$
	such that $2^N\le n+1<2^{N+1}$.
	Then we have
	\begin{multline*}
		B\le\biggl(\sum_{\mu=2^N}^\infty a_\mu^p \mu^{p-2}\biggr)^{\theta/p}
			\le\biggl(
				\sum_{\nu=N}^\infty a_{2^\nu}^p
				\sum_{\mu=2^\nu}^{2^{\nu+1}-1}\mu^{p-2}
			\biggr)^{\theta/p}\\
		\le\Cn\biggl(
			\sum_{\nu=N}^\infty a_{2^\nu}^p 2^{\nu(p-1)}
		\biggr)^{\theta/p}.
	\end{multline*}
	Making use of Lemma~\ref{lm:jensen}
	we obtain
	\begin{multline*}
		B\le\lastC\sum_{\nu=N}^\infty a_{2^\nu}^\theta
			2^{\nu(\theta-\theta/p)}
		\le\Cn\sum_{\nu=N}^\infty \sum_{\mu=2^{\nu-1}}^{2^\nu-1}
			a_\mu^\theta \mu^{\theta-\theta/p-1}\\
		=\lastC\sum_{\nu=2^{N-1}}^\infty a_\nu^\theta
			\nu^{\theta-\theta/p-1}
		\le\lastC\sum_{\nu=\left[\frac{n+1}4\right]}^\infty
			a_\nu^\theta \nu^{\theta-\theta/p-1}\\
		\le\lastC\left[\frac{n+1}4\right]^{-r\theta}
			\sum_{\nu=\left[\frac{n+1}4\right]}^\infty a_\nu^\theta
			\nu^{r\theta+\theta-\theta/p-1}.
	\end{multline*}
	Since for $n\ge3$ holds $\left[\frac{n+1}4\right]\ge\frac n{12}$,
	we get
	\begin{displaymath}
		B\le\Cn n^{-r\theta}
			\sum_{\nu=\left[\frac{n+1}4\right]}^\infty a_\nu^\theta
			\nu^{r\theta+\theta-\theta/p-1}.
	\end{displaymath}

	This way, for $0<\frac\theta p<\infty$ we proved that
	\begin{displaymath}
		B\le\Cn n^{-r\theta}
			\sum_{\nu=\left[\frac{n+1}4\right]}^\infty a_\nu^\theta
			\nu^{r\theta+\theta-\theta/p-1}.
	\end{displaymath}
	Hence~\eqref{eq:A2} yields
	\begin{displaymath}
		A_2
		\le\Cn\biggl(
			\sum_{\nu=1}^n a_\nu^\theta
				\nu^{(r+\lambda)\theta+\theta-\theta/p-1}
			+n^{\lambda\theta}\sum_{\nu=\left[\frac{n+1}4\right]}^\infty
				a_\nu^\theta \nu^{r\theta+\theta-\theta/p-1}
		\biggr).
	\end{displaymath}

	Now, from~\eqref{eq:A1} it follows that
	\begin{multline}\label{eq:J3-J4}
		J_3+J_4=n^{-\lambda\theta}(A_1+A_2)\\
		\le\Cn\biggl(
				n^{-\lambda\theta}\sum_{\nu=1}^n a_\nu^\theta
					\nu^{(r+\lambda)\theta+\theta-\theta/p-1}
				+\sum_{\nu=\left[\frac{n+1}4\right]}^\infty
					a_\nu^\theta \nu^{r\theta+\theta-\theta/p-1}
		\biggr).
	\end{multline}

	Further, we estimate the series
	\begin{displaymath}
		A_3=\sum_{\nu=\left[\frac{n+1}4\right]}^\infty
			a_\nu^\theta \nu^{r\theta+\theta-\theta/p-1}
		=A_4+\sum_{\nu=n+1}^\infty	a_\nu^\theta
			\nu^{r\theta+\theta-\theta/p-1},
	\end{displaymath}
	where is
	\begin{multline*}
		A_4=\sum_{\nu=\left[\frac{n+1}4\right]}^n
				a_\nu^\theta \nu^{r\theta+\theta-\theta/p-1}
			\le\Cn a_{\left[\frac{n+1}4\right]}^\theta
				n^{r\theta+\theta-\theta/p}\\
		\le\Cn n^{-\lambda\theta}
				\sum_{\nu=1}^{\left[\frac{n+1}4\right]} a_\nu^\theta
					\nu^{(r+\lambda)\theta+\theta-\theta/p-1}
		\le\lastC n^{-\lambda\theta}
				\sum_{\nu=1}^n a_\nu^\theta
					\nu^{(r+\lambda)\theta+\theta-\theta/p-1}.
	\end{multline*}
	Hence
	\begin{equation}\label{eq:A3}
		A_3\le\Cn\biggl(
			n^{-\lambda\theta}\sum_{\nu=1}^n a_\nu^\theta
				\nu^{(r+\lambda)\theta+\theta-\theta/p-1}
			+\sum_{\nu=n+1}^\infty	a_\nu^\theta
				\nu^{r\theta+\theta-\theta/p-1}
		\biggr).
	\end{equation}

	Making use of~\eqref{eq:A3} and~\eqref{eq:J3-J4}
	we have
	\begin{displaymath}
		J_3+J_4\le\Cn\biggl(
			n^{-\lambda\theta}\sum_{\nu=1}^n a_\nu^\theta
				\nu^{(r+\lambda)\theta+\theta-\theta/p-1}
			+\sum_{\nu=n+1}^\infty	a_\nu^\theta
				\nu^{r\theta+\theta-\theta/p-1}
		\biggr).
	\end{displaymath}
	Hence, applying~\eqref{eq:A3} in~\eqref{eq:J1-J2}
	we obtain
	\begin{multline}\label{eq:J1-J2-J3-J4le}
		J_1+J_2+J_3+J_4\\
		\le\Cn\biggl(
			n^{-\lambda\theta}\sum_{\nu=1}^n a_\nu^\theta
				\nu^{(r+\lambda)\theta+\theta-\theta/p-1}
			+\sum_{\nu=n+1}^\infty	a_\nu^\theta
				\nu^{r\theta+\theta-\theta/p-1}
		\biggr).
	\end{multline}

	Now we estimate~$A_1$ and~$A_2$ from below.
	Making use of Lemma~\ref{lm:lp-asymp}
	we get
	\begin{displaymath}
		A_1\ge\Cn\sum_{\nu=1}^n a_\nu^\theta
			\nu^{(r+\lambda)\theta+\theta-\theta/p-1},
	\end{displaymath}
	and in an analogous way
	\begin{displaymath}
		A_2\ge\sum_{\nu=1}^n \nu^{(r+\lambda)\theta-1}
			\biggl(
				\sum_{\mu=\nu}^n a_\mu^p \mu^{p-2}
			\biggr)^{\theta/p}
		\ge\Cn\sum_{\nu=1}^n a_\nu^\theta
			\nu^{(r+\lambda)\theta+\theta-\theta/p-1}.
	\end{displaymath}
	Hence
	\begin{displaymath}
		A_1+A_2\ge\Cn\sum_{\nu=1}^n a_\nu^\theta
			\nu^{(r+\lambda)\theta+\theta-\theta/p-1}.
	\end{displaymath}
	This way the following inequality holds
	\begin{displaymath}
		J_3+J_4\ge\Cn n^{-\lambda\theta}\sum_{\nu=1}^n a_\nu^\theta
			\nu^{(r+\lambda)\theta+\theta-\theta/p-1}.
	\end{displaymath}

	From~\eqref{eq:J1-J2} it follows that
	\begin{multline}\label{eq:J1-J2-J3-J4ge}
		J_1+J_2+J_3+J_4\\
		\ge\Cn\biggl(
			\sum_{\nu=8(n+1)}^\infty a_\nu^\theta
				\nu^{r\theta+\theta-\theta/p-1}
			+n^{-\lambda\theta}\sum_{\nu=1}^n a_\nu^\theta
				\nu^{(r+\lambda)\theta+\theta-\theta/p-1}
		\biggr).
	\end{multline}
	Since
	\begin{multline*}
		\sum_{\nu=n+1}^{\nu=8(n+1)-1} a_\nu^\theta
				\nu^{r\theta+\theta-\theta/p-1}
			\le\Cn a_n^\theta	n^{r\theta+\theta-\theta/p}\\
		\le\Cn n^{-\lambda\theta}\sum_{\nu=1}^n a_\nu^\theta
				\nu^{(r+\lambda)\theta+\theta-\theta/p-1}
	\end{multline*}
	holds,
	we have
	\begin{multline*}
		\sum_{\nu=n+1}^\infty a_\nu^\theta
				\nu^{r\theta+\theta-\theta/p-1}
			+n^{-\lambda\theta}\sum_{\nu=1}^n a_\nu^\theta
				\nu^{(r+\lambda)\theta+\theta-\theta/p-1}\\
		\le\Cn\biggl(
			\sum_{\nu=8(n+1)}^\infty a_\nu^\theta
				\nu^{r\theta+\theta-\theta/p-1}
			+n^{-\lambda\theta}\sum_{\nu=1}^n a_\nu^\theta
				\nu^{(r+\lambda)\theta+\theta-\theta/p-1}
		\biggr).
	\end{multline*}

	Now,
	estimates~\eqref{eq:J1-J2-J3-J4ge} and~\eqref{eq:J1-J2-J3-J4le}
	imply
	\begin{multline*}
		\Cn\biggl(
			\sum_{\nu=n+1}^\infty a_\nu^\theta
				\nu^{r\theta+\theta-\theta/p-1}
			+n^{-\lambda\theta}\sum_{\nu=1}^n a_\nu^\theta
				\nu^{(r+\lambda)\theta+\theta-\theta/p-1}
		\biggr)\\
		\le J_1+J_2+J_3+J_4\\
		\le\Cn\biggl(
			\sum_{\nu=n+1}^\infty a_\nu^\theta
				\nu^{r\theta+\theta-\theta/p-1}
			+n^{-\lambda\theta}\sum_{\nu=1}^n a_\nu^\theta
				\nu^{(r+\lambda)\theta+\theta-\theta/p-1}
		\biggr).
	\end{multline*}

	This way we proved that condition~\eqref{eq:Np-w}
	is equivalent to the condition of the theorem.
	Since condition~\eqref{eq:Np-w}
	is equivalent to the condition $f\in\Np$,
	proof of Theorem~\ref{th:Np-monotone} is completed.
\end{proof}

\begin{proof}[Proof of Theorem~\ref{th:Np-lacunary}]
	Considering Lemma~\ref{lm:Np-E}, condition $f\in\Np$ is equivalent
	to the condition
	\begin{displaymath}
		\sum_{\nu=n+1}^\infty 2^{\nu r\theta}\En[2^\nu]^\theta
			+2^{-n\lambda\theta}
				\sum_{\nu=0}^n 2^{\nu(r+\lambda)\theta}\En[2^\nu]^\theta
		\le\Cn\varphi\left(\frac 1{2^n}\right)^\theta,
	\end{displaymath}
	where constant~$C$ does not depend on~$n$.
	Corollary~\ref{cr:zygmund} yields that the last estimate is
	equivalent to the estimate
	\begin{equation}\label{eq:condition}
		\sum_{\nu=n+1}^\infty 2^{\nu r\theta}
				\biggl(\sum_{\mu=\nu}^\infty a_\mu^2\biggr)^{\theta/2}
			+2^{-n\lambda\theta}\sum_{\nu=0}^n 2^{\nu(r+\lambda)\theta}
				\biggl(\sum_{\mu=\nu}^\infty a_\mu^2\biggr)^{\theta/2}
		\le\Cn\varphi\left(\frac 1{2^n}\right)^\theta,
	\end{equation}
	where constant~$\lastC$ does not depend on~$n$.

	Put
	\begin{displaymath}
		J_1=\sum_{\nu=n+1}^\infty 2^{\nu r\theta}
				\biggl(\sum_{\mu=\nu}^\infty a_\mu^2\biggr)^{\theta/2},
		\quad
		J_2=2^{-n\lambda\theta}\sum_{\nu=0}^n 2^{\nu(r+\lambda)\theta}
				\biggl(\sum_{\mu=\nu}^\infty a_\mu^2\biggr)^{\theta/2},
	\end{displaymath}
	we estimate~$J_1$ and~$J_2$ from below and above.

	Let $0<\frac\theta2\le1$. Using Lemma~\ref{lm:jensen},
	changing the order os summation we get
	\begin{displaymath}
		J_1\le\sum_{\nu=n+1}^\infty 2^{\nu r\theta}
					\sum_{\mu=\nu}^\infty a_\mu^\theta
		=\sum_{\mu=n+1}^\infty a_\mu^\theta
				\sum_{\nu=n+1}^\mu 2^{\nu r\theta}.
	\end{displaymath}
	Therefrom, taking into consideration that $r\theta>0$ while
	computing the second sum we obtain
	\begin{displaymath}
		J_1\le\Cn\sum_{\mu=n+1}^\infty a_\mu^\theta2^{\mu r\theta}.
	\end{displaymath}

	Let $1\le\frac\theta2<\infty$ and $0<\varepsilon<r$. Applying
	H\"older inequality we have
	\begin{displaymath}
		A=\sum_{\mu=\nu}^\infty a_\mu^2
		\le\biggl(
				\sum_{\mu=\nu}^\infty a_\mu^\theta 2^{\mu\varepsilon\theta}
			\biggr)^{2/\theta}
			\biggl(
				\sum_{\mu=\nu}^\infty 2^{-2\mu\varepsilon\theta'}
			\biggr)^{1/\theta'},
	\end{displaymath}
	where is $\frac2\theta+\frac1{\theta'}=1$. Computing the
	second sum we obtain
	\begin{displaymath}
		A\le\frac\Cn{2^{2\varepsilon\nu}}
			\biggl(
				\sum_{\mu=\nu}^\infty a_\mu^\theta 2^{\mu\varepsilon\theta}
			\biggr)^{2/\theta}.
	\end{displaymath}
	Now we have
	\begin{multline*}
		J_1\le\Cn\sum_{\nu=n+1}^\infty 2^{\nu(r-\varepsilon)\theta}
				\sum_{\mu=\nu}^\infty a_\mu^\theta 2^{\mu\varepsilon\theta}\\
		=\lastC\sum_{\mu=n+1}^\infty a_\mu^\theta 2^{\mu\varepsilon\theta}
				\sum_{\nu=n+1}^\mu 2^{\nu(r-\varepsilon)\theta}
			\le\Cn\sum_{\mu=n+1}^\infty a_\mu^\theta 2^{\mu r\theta}.
	\end{multline*}

	This way, for $0<\theta<\infty$ we have
	\begin{displaymath}
		J_1\le\Cn\sum_{\mu=n+1}^\infty a_\mu^\theta 2^{\mu r\theta},
	\end{displaymath}
	where constant~$\lastC$ does not depend on~$n$.

	Now we estimate~$J_1$ from below.

	Let $1\le\frac\theta2<\infty$. Making use of Lemma~\ref{lm:jensen}
	we get
	\begin{displaymath}
		J_1\ge\sum_{\nu=n+1}^\infty 2^{\nu r\theta}
					\sum_{\mu=\nu}^\infty a_\mu^\theta
		=\sum_{\mu=n+1}^\infty a_\mu^\theta
				\sum_{\nu=n+1}^\mu 2^{\nu r\theta}.
	\end{displaymath}
	Computing the second sum we get
	\begin{displaymath}
		J_1\ge\Cn\sum_{\mu=n+1}^\infty a_\mu^\theta 2^{\mu r\theta}.
	\end{displaymath}

	Let $0<\frac\theta2\le1$ and $\varepsilon>0$. Applying H\"older
	inequality we have
	\begin{displaymath}
		\sum_{\mu=\nu}^\infty a_\mu^\theta 2^{-\mu\varepsilon\theta}
		\le\biggl(\sum_{\mu=\nu}^\infty a_\mu^2\biggr)^{\theta/2}
			\biggl(
				\sum_{\mu=\nu}^\infty 2^{-\mu\varepsilon\theta\theta'}
			\biggr)^{1/\theta'}
		\le\frac\Cn{2^{\nu\varepsilon\theta}}
			\biggl(\sum_{\mu=\nu}^\infty a_\mu^2\biggr)^{\theta/2},
	\end{displaymath}
	where is $\frac\theta2+\frac1{\theta'}=1$. The last estimate
	implies
	\begin{displaymath}
		J_1\ge\Cn\sum_{\nu=n+1}^\infty 2^{\nu(r+\varepsilon)\theta}
			\sum_{\mu=\nu}^\infty a_\mu^\theta 2^{-\mu\varepsilon\theta}.
	\end{displaymath}
	Changing the order of summation and then computing the second
	sum we obtain
	\begin{displaymath}
		J_1\ge\lastC
			\sum_{\mu=n+1}^\infty a_\mu^\theta 2^{-\mu\varepsilon\theta}
			\sum_{\nu=n+1}^\mu 2^{\nu(r+\varepsilon)\theta}
		\ge\Cn\sum_{\mu=n+1}^\infty a_\mu^\theta 2^{\mu r\theta},
	\end{displaymath}
	where constant~$\lastC$ does not depend on~$n$.

	Consequently, for every $0<\theta<\infty$ the following
	estimate holds
	\begin{equation}\label{eq:J1-lacunary}
		\Cn\sum_{\mu=n+1}^\infty a_\mu^\theta 2^{\mu r\theta}\le J_1
		\le\Cn\sum_{\mu=n+1}^\infty a_\mu^\theta 2^{\mu r\theta},
	\end{equation}
	where constants~$\prevC$ and~$\lastC$ do not depend on~$n$.

	Now we estimate~$J_2$. Obviously
	\begin{displaymath}
		J_2\ge2^{-n\lambda\theta}\sum_{\nu=0}^n 2^{\nu(r+\lambda)\theta}
				\biggl(\sum_{\mu=\nu}^n a_\mu^2\biggr)^{\theta/2}.
	\end{displaymath}

	Let $1\le\frac\theta2<\infty$. Applying Lemma~\ref{lm:jensen},
	changing the order of summation, and then computing the second
	sum we obtain
	\begin{multline*}
		J_2\ge2^{-n\lambda\theta}\sum_{\nu=0}^n 2^{\nu(r+\lambda)\theta}
			\sum_{\mu=\nu}^n a_\mu^\theta\\
		=2^{-n\lambda\theta}\sum_{\mu=0}^n a_\mu^\theta
			\sum_{\nu=0}^\mu 2^{\nu(r+\lambda)\theta}
		\ge\Cn2^{-n\lambda\theta}
			\sum_{\mu=0}^n a_\mu^\theta2^{\mu(r+\lambda)\theta}.
	\end{multline*}

	Let $0<\frac\theta2\le1$ and $\varepsilon>0$. Applying H\"older
	inequality we get
	\begin{displaymath}
		\sum_{\mu=\nu}^n a_\mu^\theta 2^{-\mu\varepsilon\theta}
		\le\biggl(\sum_{\mu=\nu}^n a_\mu^2\biggr)^{\theta/2}
			\biggl(
				\sum_{\mu=\nu}^n 2^{-\mu\varepsilon\theta\theta'}
			\biggr)^{1/\theta'}
		\le\frac\Cn{2^{\nu\varepsilon\theta}}
			\biggl(\sum_{\mu=\nu}^n a_\mu^2\biggr)^{\theta/2},
	\end{displaymath}
	where is $\frac\theta2+\frac1{\theta'}=1$. The last estimate
	implies
	\begin{displaymath}
		J_2\ge\Cn2^{-n\lambda\theta}
			\sum_{\nu=0}^n 2^{\nu(r+\lambda+\varepsilon)\theta}
			\sum_{\mu=\nu}^n a_\mu^\theta 2^{-\mu\varepsilon\theta}.
	\end{displaymath}
	Changing the order of summation and computing the second
	sum we have
	\begin{displaymath}
		J_2\ge\lastC2^{-n\lambda\theta}
			\sum_{\mu=0}^n a_\mu^\theta 2^{-\mu\varepsilon\theta}
			\sum_{\nu=0}^\mu 2^{\nu(r+\lambda+\varepsilon)\theta}
		\ge\Cn2^{-n\lambda\theta}
			\sum_{\mu=0}^n a_\mu^\theta 2^{\mu(r+\lambda)\theta}.
	\end{displaymath}

	Thus, for every $0<\theta<\infty$ holds
	\begin{equation}\label{eq:J2-ge}
		J_2\ge\Cn2^{-n\lambda\theta}
			\sum_{\mu=0}^n a_\mu^\theta 2^{\mu(r+\lambda)\theta}.
	\end{equation}

	Now we estimate~$J_2$ from above. Taking into consideration
	that $(r+\lambda)\theta>0$, we
	have
	\begin{multline}\label{eq:J2-le}
		J_2\le\Cn2^{-n\lambda\theta}\sum_{\nu=0}^n 2^{\nu(r+\lambda)\theta}
			\Bigg(
				\biggl(\sum_{\mu=\nu}^n a_\mu^2\biggr)^{\theta/2}
				+\biggl(\sum_{\mu=n+1}^\infty a_\mu^2\biggr)^{\theta/2}
			\Bigg)\\
		\le\Cn\Biggl(
				2^{-n\lambda\theta}\sum_{\nu=0}^n 2^{\nu(r+\lambda)\theta}
					\biggl(\sum_{\mu=\nu}^n a_\mu^2\biggr)^{\theta/2}
				+2^{nr\theta}
					\biggl(\sum_{\mu=n+1}^\infty a_\mu^2\biggr)^{\theta/2}
			\Biggr).
	\end{multline}
	Since
	\begin{displaymath}
		2^{nr\theta}
			\biggl(\sum_{\mu=n+1}^\infty a_\mu^2\biggr)^{\theta/2}
		\le\sum_{\mu=n+1}^\infty 2^{\nu r\theta}
			\biggl(\sum_{\mu=n+1}^\infty a_\mu^2\biggr)^{\theta/2}
		=J_1
	\end{displaymath}
	holds and an upper bound for~$J_1$ is already found, we
	estimate from above the expression
	\begin{displaymath}
		J_3=\sum_{\nu=0}^n 2^{\nu(r+\lambda)\theta}
			\biggl(\sum_{\mu=\nu}^n a_\mu^2\biggr)^{\theta/2}.
	\end{displaymath}

	Let $0<\frac\theta2\le1$. Applying Lemma~\ref{lm:jensen}
	we obtain
	\begin{displaymath}
		J_3\le\sum_{\nu=0}^n 2^{\nu(r+\lambda)\theta}
			\sum_{\mu=\nu}^n a_\mu^\theta
		=\sum_{\mu=0}^n a_\mu^\theta
			\sum_{\nu=0}^\mu 2^{\nu(r+\lambda)\theta}
		\le\Cn\sum_{\mu=0}^n a_\mu^\theta 2^{\mu(r+\lambda)\theta}.
	\end{displaymath}

	Let $1\le\frac\theta2<\infty$ and $0<\varepsilon<r+\lambda$.
	Then applying H\"older inequality we have
	\begin{displaymath}
		\sum_{\mu=\nu}^n a_\mu^2
		\le\biggl(
				\sum_{\mu=\nu}^n a_\mu^\theta 2^{\mu\varepsilon\theta}
			\biggr)^{2/\theta}
			\biggl(
				\sum_{\mu=\nu}^n 2^{-2\mu\varepsilon\theta'}
			\biggr)^{1/\theta'},
	\end{displaymath}
	where is $\frac2\theta+\frac1{\theta'}=1$. Using the last
	estimate we get
	\begin{displaymath}
		J_3\le\sum_{\nu=0}^n 2^{\nu(r+\lambda)\theta}
			\biggl(
				\sum_{\mu=\nu}^\infty 2^{-2\mu\varepsilon\theta'}
			\biggr)^{\frac\theta{2\theta'}}
			\sum_{\mu=\nu}^n a_\mu^\theta 2^{\mu\varepsilon\theta}
		\le\Cn\sum_{\nu=0}^n 2^{\nu(r+\lambda-\varepsilon)\theta}
			\sum_{\mu=\nu}^n a_\mu^\theta 2^{\mu\varepsilon\theta}.
	\end{displaymath}
	Changing the order of summation and computing the second
	sum we obtain
	\begin{displaymath}
		J_3\le\lastC\sum_{\mu=0}^n a_\mu^\theta 2^{\mu\varepsilon\theta}
			\sum_{\nu=0}^\mu 2^{\nu(r+\lambda-\varepsilon)\theta}
		\le\Cn\sum_{\mu=0}^n a_\mu^\theta 2^{\mu(r+\lambda)\theta}.
	\end{displaymath}

	Therefore, for every $0<\theta<\infty$ the following
	estimate holds
	\begin{displaymath}
		J_3\le\Cn\sum_{\mu=0}^n a_\mu^\theta 2^{\mu(r+\lambda)\theta}.
	\end{displaymath}

	Now making use of inequalities~\eqref{eq:J2-le}
	and~\eqref{eq:J1-lacunary} we have
	\begin{displaymath}
		J_2\le\Cn\biggl(
				2^{-n\lambda\theta}
					\sum_{\mu=0}^n a_\mu^\theta 2^{\mu(r+\lambda)\theta}
				+\sum_{\mu=n+1}^\infty a_\mu^\theta 2^{\mu r\theta}
			\biggr).
	\end{displaymath}

	This way, inequalities~\eqref{eq:J1-lacunary}, \eqref{eq:J2-ge} and
	the last inequality imply the estimate
	\begin{multline*}
		\Cn\biggl(
				\sum_{\mu=n+1}^\infty a_\mu^\theta 2^{\mu r\theta}
				+2^{-n\lambda\theta}
					\sum_{\mu=0}^n a_\mu^\theta 2^{\mu(r+\lambda)\theta}
			\biggr)\le J_1+J_2\\
		\le\Cn\biggl(
				\sum_{\mu=n+1}^\infty a_\mu^\theta 2^{\mu r\theta}
				+2^{-n\lambda\theta}
					\sum_{\mu=0}^n a_\mu^\theta 2^{\mu(r+\lambda)\theta}
			\biggr),
	\end{multline*}
	where constants~$\prevC$ and~$\lastC$ do not depend on~$n$.
	Hence, considering the condition~\eqref{eq:condition} we
	conclude that condition $f\in\Np$ is equivalent to the
	condition
	\begin{equation}\label{eq:An}
		A_n=\sum_{\mu=n+1}^\infty a_\mu^\theta 2^{\mu r\theta}
				+2^{-n\lambda\theta}
					\sum_{\mu=0}^n a_\mu^\theta 2^{\mu(r+\lambda)\theta}
		\le\Cn\varphi\left(\frac1{2^n}\right)^\theta,
	\end{equation}
	where constant~$\lastC$ does not depend on~$n$.

	We put
	\begin{displaymath}
		D_m=\sum_{\nu=m+1}^\infty \lambda_\nu^\theta \nu^{r\theta}
			+m^{-\lambda\theta}
				\sum_{\nu=1}^m \lambda_\nu^\theta \nu^{(r+\lambda)\theta}.
	\end{displaymath}
	For given~$m$ we choose the positive integer~$n$ such that
	$2^n\le m+1<2^{n+1}$.

	First we consider the case $2^n<m+1<2^{n+1}$. We have
	\begin{multline*}
		D_m=\sum_{\nu=2^{n+1}}^\infty \lambda_\nu^\theta \nu^{r\theta}
			+\sum_{\nu=m+1}^{2^{n+1}-1} \lambda_\nu^\theta \nu^{r\theta}
			+m^{-\lambda\theta}
				\sum_{\nu=1}^{2^n-1}
					\lambda_\nu^\theta \nu^{(r+\lambda)\theta}\\
		+m^{-\lambda\theta}
			\sum_{\nu=2^n}^m \lambda_\nu^\theta \nu^{(r+\lambda)\theta}.
	\end{multline*}
	Since $\lambda_\nu=0$ for $\nu\ne2^\mu$, we get
	\begin{multline*}
		D_m=\sum_{\nu=2^{n+1}}^\infty \lambda_\nu^\theta \nu^{r\theta}
			+m^{-\lambda\theta}
				\sum_{\nu=1}^{2^n-1}
					\lambda_\nu^\theta \nu^{(r+\lambda)\theta}
			+m^{-\lambda\theta}
				\lambda_{2^n}^\theta 2^{n(r+\lambda)\theta}\\
		=\sum_{\mu=n+1}^\infty \sum_{\nu=2^\mu}^{2^{\mu+1}-1}
				\lambda_\nu^\theta \nu^{r\theta}
			+m^{-\lambda\theta}\sum_{\mu=0}^{n-1}
				\sum_{\nu=2^\mu}^{2^{\mu+1}-1}
				\lambda_\nu^\theta \nu^{(r+\lambda)\theta}
			+m^{-\lambda\theta} \lambda_{2^n}^\theta 2^{n(r+\lambda)\theta}\\
		=\sum_{\mu=n+1}^\infty \lambda_{2^\mu}^\theta 2^{\mu r\theta}
			+m^{-\lambda\theta}\sum_{\mu=0}^{n-1}
				\lambda_{2^\mu}^\theta 2^{\mu(r+\lambda)\theta}
			+m^{-\lambda\theta} \lambda_{2^n}^\theta 2^{n(r+\lambda)\theta}.
	\end{multline*}
	Further, since $\lambda_{2^\mu}=a_\mu$, we get
	\begin{displaymath}
		D_m=\sum_{\mu=n+1}^\infty a_\mu^\theta 2^{\mu r\theta}
			+m^{-\lambda\theta}
				\sum_{\mu=0}^n a_\mu^\theta 2^{\mu(r+\lambda)\theta}.
	\end{displaymath}
	Hence, for $2^n<m+1<2^{n+1}$ we obtain
	\begin{multline*}
		\Cn\biggl(
				\sum_{\mu=n+1}^\infty a_\mu^\theta 2^{\mu r\theta}
				+2^{-n\lambda\theta}
					\sum_{\mu=0}^n a_\mu^\theta 2^{\mu(r+\lambda)\theta}
			\biggr)\le D_m\\
		\le\Cn\biggl(
				\sum_{\mu=n+1}^\infty a_\mu^\theta 2^{\mu r\theta}
				+2^{-n\lambda\theta}
					\sum_{\mu=0}^n a_\mu^\theta 2^{\mu(r+\lambda)\theta}
			\biggr),
	\end{multline*}
	where constants~$\prevC$ and~$\lastC$ do not depend on~$m$
	and~$n$.

	Let us assume now that $m+1=2^n$. In an analogous way we have
	\begin{multline*}
		D_m=\sum_{\nu=2^n}^\infty \lambda_\nu^\theta \nu^{r\theta}
				+2^{-n\lambda\theta}
					\sum_{\nu=1}^{2^n-1} \lambda_\nu^\theta
					  \nu^{(r+\lambda)\theta}\\
		=\sum_{\mu=n}^\infty a_\mu^\theta 2^{\mu r\theta}
				+2^{-n\lambda\theta}
					\sum_{\mu=0}^{n-1} a_\mu^\theta 2^{\mu(r+\lambda)\theta}\\
		=\sum_{\mu=n+1}^\infty a_\mu^\theta 2^{\mu r\theta}
				+2^{-n\lambda\theta}
					\sum_{\mu=0}^n a_\mu^\theta 2^{\mu(r+\lambda)\theta}
			=A_n.
	\end{multline*}

	Thus, for $2^n\le m+1<2^{n+1}$ the following estimate holds
	\begin{displaymath}
		\Cn A_n\le D_m\le\Cn A_n,
	\end{displaymath}
	where constants~$\prevC$ and~$\lastC$ do not depend on~$m$
	and~$n$. Hence, considering the condition~\eqref{eq:An}
	we conclude that condition $f\in\Np$ is equivalent to the
	condition
	\begin{equation}\label{eq:Dm}
		D_m\le\Cn\varphi\left(\frac1{2^n}\right)^\theta,
	\end{equation}
	where constant~$\lastC$ does not depend on~$m$ and~$n$.

	Since $\frac1{2^n}<\frac2{m+1}<\frac2m$, we get
	\begin{displaymath}
		\varphi\left(\frac1{2^n}\right)\le\Cn\varphi\left(\frac2{m+1}\right)
		\le\Cn\varphi\left(\frac2m\right),
	\end{displaymath}
	where constant~$\lastC$ does not depend on~$m$ and~$n$;
	and since $\frac1{2^n}\ge\frac1{m+1}\ge\frac1{2m}$ we get
	\begin{displaymath}
		\varphi\left(\frac1{2^n}\right)\ge\Cn\varphi\left(\frac1{2m}\right)
		\ge\Cn\varphi\left(\frac1m\right),
	\end{displaymath}
	where constant~$\lastC$ does not depend on~$m$ and~$n$.
	This way, condition~\eqref{eq:Dm} is equivalent to the
	condition
	\begin{displaymath}
		D_m\le\Cn\varphi\left(\frac1m\right)^\theta,
	\end{displaymath}
	where constant~$\lastC$ does not depend on~$m$.

	This completes the proof of Theorem~\ref{th:Np-lacunary}.
\end{proof}

\begin{remark}
	Notice that another way of proving Theorems~\ref{th:Np-monotone}
	and~\ref{th:Np-lacunary}
	is presented in~\cite{tikhonov:etna-05}.
	Our approach here is similar to that used in~\cite{b-berisha:math-99}.
\end{remark}

\bibliographystyle{amsplain}
\bibliography{maths}

\end{document}